\documentclass[10pt]{amsart}
\usepackage[dvips]{graphicx}
\usepackage{bbm}
\usepackage{epsfig}
\usepackage{caption}
\usepackage{color}
\usepackage{amsthm}
\usepackage{amsmath}
\usepackage{amsfonts}
\usepackage{amssymb}
\usepackage{graphics,graphicx}
\usepackage{grffile}

\usepackage{color}
\def\be{\begin{equation}}
\def\ee{\end{equation}}
\def\beq{\begin{eqnarray}}
\def\eeq{\end{eqnarray}}
\def\beqs{\begin{eqnarray*}}
\def\eeqs{\end{eqnarray*}}
\def\ea{\end{array}}
\def\ea{\end{array}}
\def\ds{\displaystyle}

\def \AAA{{\mathcal A}_{\alpha}^*}
\def \AA{{\mathcal A}_{\alpha}}
\def\11{{\rm 1~\hspace{-1.5ex}1} }
\def\NN{\mathbb{N}}
\def\ZZ{\mathbb{Z}}
\def\CC{\mathbb{C}}

\newcommand{\rfb}[1]{\mbox{\rm
   (\ref{#1})}\ifx\undefined\stillediting\else:\fbox{$#1$}\fi}

\makeatletter
\def\section{\@startsection {section}{1}{\z@}{-3.5ex plus -1ex minus
    -.2ex}{2.3ex plus .2ex}{\large\bf}}
\makeatother

\def\DD {{\mathcal D}}

\font\eufm=eufm10\font\eufms=eufm10\font\eufmss=eufm10\newfam\eufam
\textfont\eufam=\eufm\scriptfont\eufam=\eufms\scriptscriptfont\eufam=\eufmss

\usepackage{amsthm}
\swapnumbers
\newtheorem{theorem}{Theorem}[section]
\newtheorem{lemma}[theorem]{Lemma}
\newtheorem{corollary}[theorem]{Corollary}
\newtheorem{remark}[theorem]{Remark}

\newtheorem{proposition}[theorem]{Proposition}

\begin{document}
\thispagestyle{empty}
\title[Wave operator on the tadpole graph]{Spectral analysis and best decay rate of the wave propagator on the tadpole graph}

\author{Ka\"{\i}s Ammari}
\address{LR Analysis and Control of PDEs, LR22ES03, Department of Mathematics, Faculty of Sciences of
Monastir, University of Monastir, 5019 Monastir, Tunisia}
\email{kais.ammari@fsm.rnu.tn}

\author{Rachid Assel}
\address{LR Analysis and Control of PDEs, LR22ES03, Department of Mathematics, Faculty of Sciences of
Monastir, University of Monastir, 5019 Monastir, Tunisia}
\email{rachid.assel@fsm.rnu.tn}

\author{Mouez Dimassi}
\address{Universit\'e Bordeaux, CNRS, UMR 5251 IMB, 351, Cours de la Libration, 33405 Talence Cedex, France } 
\email{mdimassi@u-bordeaux.fr}

\date{}

\begin{abstract} 
We consider the damped wave semigroup
on the tadpole graph ${\mathcal R}$. 
We first give a meticulous spectral analysis, followed by a judicious decomposition of the resolvent's kernel.
 As a consequence,  and by showing that the generalized eigenfunctions form a Riesz basis of some subspace of the energy space $\mathcal{H}$, 
we establish the exponential decay of the corresponding energy, with the optimal decay rate dictated by the spectral abscissa of the relevant operator.

\end{abstract}


\maketitle

\tableofcontents

\section{Introduction} \label{formulare}


Over the past few years, there has been a burgeoning interest in various physical models encompassing multi-link flexible structures. These structures comprise either finitely or infinitely many interconnected flexible elements, such as strings, beams, plates, and shells. A wealth of literature has emerged on this subject, underscoring its significance. Notable references include \cite{BK,amregmer} and \cite{nojaetall:15, cacc, kost}, along with the extensive bibliographies provided therein. The spectral analysis of these intricate structures holds promise for addressing control or stabilization issues, as discussed in works such as \cite{ammarinicaise}.

\begin{center}\label{fig1} 
\includegraphics[scale=0.80]{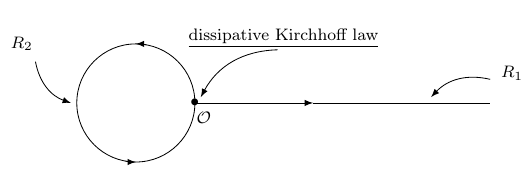}
\captionof{figure}{Tadpole graph}
\end{center}

A tadpole graph is a term commonly used in mathematics and graph theory. It refers to a specific type of graph that resembles the shape of a tadpole, hence the name.  In a tadpole graph, you have a main cycle (the body of the tadpole) with additional nodes (the tail) attached to it. A tadpole graph is a  quantum graph with a simple geometry, which has been extensively researched in the past decade (see \cite{amamnic1} and references therein) and has numerous practical applications.  Dispersive effects for the  free Schr\"odinger equation on the tadpole
was studied in \cite{amamnic1}. The case of the damped 
 Schr\"odinger semigroup was examined in \cite{amaass1}. For the nonlinear Schr\"odinger equation on the tadpole we refer to  \cite{cacc,AD1,AD2} and the references given there.

The damped wave equation was intensively studied in the case of compact graph, where mostly controllability, observability and stabilization of the problem are studied, see e.g. \cite{DZ,AS}. Very little was known about spectral properties of the wave equation on non-compact graphs (see \cite{AJK,AG1,AG2}).  To our best knowledge the damped wave equation  on a tadpole graph (or any other partial differential equation  on  non compact graphs with cycles) has not been investigated.

\medskip

In the present study, we consider the damped wave equation on a non-compact tadpole graph (with Kirchhoff condition  at the single vertex incorporating a damping
term $-\alpha u_t$ see \eqref{wsystem}). We conduct a meticulous spectral analysis of the dissipative wave operator acting on the tadpole graph. The main result of the paper describes the spectrum of the generator, $\AA$, and the asymptotic behavior of $e^{t\AA}$ as $t\rightarrow +\infty$. First, we prove that the spectrum of $\AA$ depends on two exceptional values $\alpha=1, 3.$
For $\alpha \in ]0, 1]\cup \{3\}$, the spectrum coincides with the imaginary axis $i\mathbb R$ with a sequence of embedded eigenvalues,  as in the damped free case $\alpha=0$. However, for $\alpha\in ]1, 3[\cup ]3, +\infty[$ a discrete spectrum  arises in the left half complex plan. 
Second, we explore an application of our analysis by investigating the best decay rate of the corresponding stabilization problem.  



\section{Preliminaries and notations}
In  this section  we review background regarding Sobolev spaces, we introduce the model  and  state our main results.

Let $R_i,i=1,2,$ be two disjoint sets
identified with a closed path of measure equal to $L > 0$ for $R_2$ and to $(0,+\infty),$
for $R_1$, see Figure \ref{fig1}.   Here, we normalize the interval for the circle to $[0, L]$ with the end points connected to the half-line $[0, +\infty[$ at a single vertex.

We set ${\mathcal R} := \ds \cup_{k=1}^2
\overline{R}_k$. We denote by $f = (f_k)_{k=1,2} =
(f_1,f_2)$ the functions on ${\mathcal R}$ taking their values in
$\CC$ and let $f_k$ be the restriction of $f$ to $R_k$.

We use three different kinds of Sobolev
spaces on  $R_i$.  For $m\in \mathbb N$, we denote by $H^m(R_i)$ the Sobolev space of functions in $L^2(R_i)$
whose partial derivatives up to $m$-th order are in $L^2(R_i)$. We denote by $\dot{H}^m(R_i)$  the homogeneous Sobolev space of order $m$, that is the set 
of functions $u\in L^2_{{\rm loc}}(R_i)$ such that $u^{(m)}\in L^2(R_i)$\footnote{Since $\Vert u^{(m)}\Vert_{L^2(R_i)}=0$ holds for any polynomial of degree $\leq m-1$, we need to consider the quotient of ${\dot H}^m$ by polynomial of degree $\leq m-1$ with the same norm to make it an Hilbert space.}. Finally,  for $m\not =0$,  ${\widehat H}^m$ denotes the  quotient space $\dot H^m/{\mathcal P}_{m-1}$, where $\mathcal P_m$ is the set of polynomial of degree $\leq m$.  
We introduce the standard inner product in $H^m$ and  $ {\widehat H}^m$ :
$$\langle f,g\rangle_{H^m( R_i)}=\sum_{l=0}^m\langle f^{(l)}, g^{(l)}\rangle_{L^2(R_i)},\,\,\,\,\, \langle f,g\rangle_{ \widehat {H}^m( R_i)}=\langle f^{(m)}, g^{(m)}\rangle_{L^2(R_i)}, \, i=1,2.$$
Next let us recall the  Hilbert  Sobolev spaces on $\mathcal R$ :

$$
L^2 (\mathcal{R}) =\oplus_{k=1}^2L^2(R_k),$$

$$ H^1(\mathcal{R)} = \left\{f=(f_k)_{k=1,2}\in \oplus_{k=1}^2 H^1(R_k);  f_1(0)=f_2(0)=f_2(L)\right\},$$
$$ \dot{H}^1(\mathcal{R)} = \left\{f=(f_k)_{k=1,2}\in \oplus_{k=1}^2 \dot{H}^1(R_k);  f_1(0)=f_2(0)=f_2(L)\right\},$$
$$ {\widehat H}^1(\mathcal R)=\dot H^1(\mathcal R)/{\mathcal P}_0,$$

and

$$
\mathcal{H} := {\widehat H}^1(\mathcal{R}) \times L^2 (\mathcal{R}).
$$
These Hilbert spaces are endowed with their natural   inner products : 
$$\langle w,v\rangle_{H^1(\mathcal R)}=\sum_{k=1}^2\Big(\langle w_k', v_k'\rangle_{L^2(R_k)}+\langle w_k, v_k\rangle_{L^2(R_k)}\Big),$$
$$\langle u,v\rangle_{L^2(\mathcal R)}=\sum_{k=1}^2\langle u_k, v_k\rangle_{L^2(R_k)},\,\,\,\,\,\,\,\,\,\,\,\,  \langle f,g\rangle_{{\widehat H}^1(\mathcal R)}=\sum_{k=1}^2\langle f_k', g_k'\rangle_{L^2(R_k)},$$
and
$$\langle (f,u), (g,v)\rangle_{\mathcal H}= \sum_{k=1}^2 \Big\{\langle f_k', g_k'\rangle+\langle u_k, v_k\rangle \Big\}.$$


 Now, let us introduce our model.  
Fix $\alpha\in \mathbb R$. For $(u, v)\in \mathcal H$, with $v_k \in C^0(\overline {R_k})$ and $u_k\in C^1(\overline{R_k})$,  we  introduce the following transmission condition
\begin{equation}\tag{A$_{\alpha}$}\label{A1}
  \sum_{k=1}^2
\frac{du_k}{dx}(0^+) - \frac{du_2}{dx}(L^-)= \alpha v_1 (0).
\end{equation}

\medskip
In this work, we are concerned with  the unbounded linear operator, $\AA$, of  ${\mathcal H}$ with domain
$$
{\mathcal D}(\mathcal{A}_\alpha) := \left\{(u,v) \in \mathcal H:\,  u, v \in {\widehat H}^1(\mathcal{R}), \,  u_k\in \dot
{H}^2(R_k) \text { and   satisfies } (A_\alpha)
\right\},
$$
defined by
 $$
\mathcal{A}_\alpha 
:= 
\left(
\begin{array}{cc}
0 & I \\
\partial^2 & 0
\end{array}
\right).
$$
Now, we consider the wave equation in the tadpole, given by the following system:  
\be
\label{wsystem}
\left\{
\begin{array}{lll}
u_i^{\prime \prime} (x,t) - \partial^2_x u_i(x,t) = 0, \, R_i \times (0,+\infty), i=1,2, \\
 \ds \sum_{k=1}^2
\partial_x u_k (0^+,t) - \partial_x u_2 (L^-,t)= \alpha u^\prime_1 (0,t), \, t > 0,\\
u_i^0(x,0) =u_i^0(x) , u_i^1(x,0) = u_i^1(x), \, R_i, i=1,2.
\end{array}
\right.
\ee
So, system (\ref{wsystem}) can be rewritten as an abstract Cauchy problem: 

\begin{equation} 
\left\{
\begin{array}{l}
U^{\prime} (t) = \mathcal{A}_\alpha U(t), \, t >0,\\
U(0)= U^{0},
\end{array}
\right.
\label{pbfirstorderch}
\end{equation}
where $U(t) := (u(t),u^\prime(t)), U^0 := (u^0,u^1).$

\medskip

We will prove according to Theorem \ref{th1}, that the operator $\mathcal{A}_\alpha$ generates a $C_0$ semigroup of contractions on $\mathcal{H}$, denoted by $(e^{t\mathcal{A}_\alpha})_{t \geq 0}$ and consequently the system (\ref{pbfirstorderch}) is well-posed and the regular solutions satisfy the following energy identity:
\be
\label{energid}
\left\|(u(t),u^\prime(t))\right\|^2_{\mathcal{H}} - \left\|(u_0,u_1)\right\|^2_{\mathcal{H}} = - 2 \alpha \, \int_0^t \left|u^\prime_1(0,s)\right|^2 \, ds, \,\forall \, t  \geq 0. 
\ee

We can now state our first main theorem.
\begin{theorem}\label{Intr}
For $\alpha\geq 0$, the spectrum of $\AA$  is given by $\sigma(\AA)=i\mathbb R\cup \Sigma_{p,\alpha}$, and consists of the absolutely continuous spectrum 
$\sigma_{\rm ac}(\AA)=i\mathbb R$, 
 a sequence of embedded eigenvalues $\Sigma_p=\{z_n:=\frac{2\pi n}{L} i; n\in \mathbb Z^*\}$, and  {\rm (}for $\alpha \in ]1, +\infty[$ with $\alpha\not=3${\rm )} a sequence of  discrete eigenvalues, $\Sigma_{p, \alpha}\subset \mathbb C^-$ {\rm (}see Theorem \ref{th1} {\rm )}.

\end{theorem}

The second main results of this paper pertain to the exponential behavior of $(e^{t\mathcal{A}_\alpha})_{t \geq 0}$, which is obtained by using a Riesz basis method. These results are outlined and demonstrated in Section \ref{asybasis} (see Theorem \ref{rieszbasis}) and Section \ref{estenerg} (see Theorem \ref{energy}), respectively.

\begin{theorem} \label{energyen} {\rm (}Riesz basis and energy estimate{\rm)} 
We assume that $\alpha \in ]1,3[ \cup ]3,+\infty[$. Let $\mathcal{H}_{p,\alpha}$ (respectively $\mathcal{H}_p$) be the subspace of $\mathcal{H}$ associated to $\Sigma_{p,\alpha}$ (resp. associated to $\Sigma_p$). 

\begin{enumerate}
\item 
The generalized eigenfunctions of the operator $\mathcal{A}_\alpha$ associated to the eigenvalues in $\Sigma_{p,\alpha}, \alpha \in ]0,1[ \cup ]3,+\infty[,$ form a Riesz basis of  $\mathcal{H}_{p,\alpha}$. 

\item $\mathcal{H}_{p,\alpha}$ is orthogonal    to $\mathcal{H}_p$. 
\item Let $(u_0,u_1)$ in $\mathcal{H}_{p,\alpha} \oplus \mathcal{H}_p$ be the initial condition of the boundary value problem (\ref{pbfirstorderch}) and $(u_0^p,u_1^p)$ its orthogonal projection onto $\mathcal{H}_p$. Then $(u(t),u^\prime(t))$ solution of (\ref{wsystem}) decreases exponentially, in $\mathcal{H}$, to 
$(u_0^p,u_1^p)$  
when $t$ tends to $+\infty$ with decay rate given by the spectral abscissa $\ds \sup_{\lambda \in \Sigma_{p,\alpha}} \Re(\lambda) < 0.$
\end{enumerate}
\end{theorem}

We end this section  with an outline of the paper.
In  section \ref{SPe}, we discuss some spectral properties of the operator $\AA$. In Proposition \ref{P0}, we prove that $\AA$ with domain ${\mathcal D}(\AA)$ is closed, dissipative (respectively accretive, respectively skew-symmetric) if $\alpha>0$,  (respectively $\alpha<0$, respectively  $\alpha=0$). The rest of this section is dedicated to proving Theorem \ref{Intr}. The proof was divided into multiple natural steps. In the first one (see Lemma \ref{P2} and Corollary \ref{Res}),  we show that there is no spectrum in the right-half plane,  and we give an optimal bound estimate for the resolvent. The Lax-Milgram theorem is used as the basis for the proof. In Theorem \ref{th1}, explicit computation provides a precise description of all the spectrum of $\AA$. Finally, to investigate the spectrum of $\AA$ 
in the left half plane, we compute explicitly the resolvent of $\AA$ (see Theorem \ref{resolv}).
In section \ref{asybasis}, we first prove that the generalized eigenfunctions associated to the discrete spectrum, $\Sigma_{p, \alpha}$ with $\alpha\in ]1,3[\cup ]3, +\infty[$ form a Riesz basis of a subspace of $\mathcal H$.  
This enables us to demonstrate Theorem \ref{energyen}.

 \vskip 1cm
 
{\bf Notations.} The following standard notations are used in the paper:
\begin{itemize}
\item $\sigma(H)$, $\sigma_c(H)$, $\sigma_p(H)$, $\rho(H)$ denote the spectrum, continuous spectrum, point spectrum and the resolvent set for an operator $H$, respectively.
\item
For  $u= (u_k)_{k=1,2} \in {\widehat H}^1(\mathcal R)$ with $u_k\in H^2(R_k)$, $\partial^2u$ (or $u^{(2)}$) denotes $\left(\frac{d^2 u_k}{dx^2}\right)_{k=1,2}$.
\item The class  $C^0(\overline {R}_k)$ consists of all continuous functions on the closed set  $\overline{R}_k$.
\item  The class $C^1(\overline{R}_k)$ consists of all differentiable functions whose derivative is continuous on $\overline{R}_k$.
\item We denote $\mathbb C^\pm=\{z\in \mathbb C; \,\,\, \pm \Re z>0\}$.
\end{itemize}

\section{Spectrum of the wave operator}\label{SPe}

 In this section we are interested in the spectral properties of the operator $\AA$. 

Throughout this paper we use frequently  the following classical lemma. For the sake of completeness we provide its proof in an Appendix.


\begin{lemma}\label{P1}
$H^1(\mathcal R)$ endowed with the inner product $\langle \cdot, \cdot \rangle_{{\widehat H}^1(\mathcal R)}$ is dense in  ${\widehat H}^1(\mathcal R)$.
\end{lemma}





\begin{proposition}\label{P0}
The operator ${\mathcal A}_\alpha$ with domain ${\mathcal D}(\mathcal{A}_\alpha)$   satisfies the following properties:
\begin{enumerate}
\item For all $U=(u,v)\in {\mathcal D}(\mathcal{A}_\alpha)$,
\begin{equation}\label{U2}
\langle {\mathcal A}_\alpha U, U\rangle=-2i\Im \sum_{k=1}^2\langle u'_k, v'_k\rangle-\alpha \vert v_1(0)\vert^2.
\end{equation}
\item $(\AA, {\mathcal D}(\AA))$ is a closed operator.
\item $\DD (\AAA)=\DD ({\mathcal A}_{-\alpha}),
\mathcal{A}_\alpha^*
:= 
\left(
\begin{array}{cc}
0 & -I \\
-\partial^2 & 0
\end{array}
\right)
$
and for all $U=(u,v)\in {\mathcal D}(\mathcal{A}^*_\alpha)$,
\begin{equation}\label{U2adj}
\langle {\mathcal A}^*_\alpha U, U\rangle= 2i\Im \sum_{k=1}^2\langle u'_k, v'_k\rangle-\alpha \vert v_1(0)\vert^2.
\end{equation}
\end{enumerate}
\end{proposition}

\begin{proof}

For $U=(u,v) \in {\mathcal D}(\mathcal{A}_\alpha)$, we have
$$
 \langle \AA U,U\rangle=\sum_{k=1}^2\Big\{\langle v'_k, u_k'\rangle+\langle u_k'', v_k\rangle \Big\}=-2i\sum_{k=1}^2\Im \langle u'_k, v'_k\rangle-\Big(u'_2(0)+u'_1(0)-u'_2(L)\Big)v_1(0), $$
which  yields \eqref{U2} by \eqref{A1}.

\medskip

Let $U(n)=(u(n),v(n))$ be  a sequence in ${\mathcal D}(\mathcal{A}_\alpha)$ such that  $U(n)$  (resp. $\AA U(n)$) converges to $(x,y)$ (resp. $(f, g)$) in 
${\mathcal H}$. 

Since $v(n)\xrightarrow{\widehat{H}^1(\mathcal R)} f$ and  $v(n)\xrightarrow {{L}^2(\mathcal R)} y$ it follows that  $y=f\in H^1(\mathcal R)$. Similarly, from
$u(n)\xrightarrow{\widehat{H}^1(\mathcal R)} x$,  $u''(n)\xrightarrow {{L}^2(\mathcal R)} g$ we deduce $x\in \widehat{H}^1(\mathcal R)$  and  $x''=g\in L^2(\mathcal R)$. By definition of $\widehat {H}^1(\mathcal R)$ we may assume that $v(n)(0)=f(0)$ for all $n\in \mathbb N$. 

Applying \eqref{U2} to $U(n)$, we obtain
$$
\langle {\mathcal A}_\alpha U(n), U(n)\rangle=-2i\Im \sum_{k=1}^2\langle u'_k(n), v'_k(n)\rangle-\alpha \vert f(0)\vert^2.
$$
Letting $n\rightarrow +\infty$, we conclude that


\begin{equation}\label{U0002}
\langle (f,g),  (x,y)\rangle_{\mathcal H}=
\sum_{k=1}^2 \Big(\langle f_k',x'_k    \rangle +\langle g_k, y_k\rangle \Big) =-2i\Im \sum_{k=1}^2\langle x'_k, y'_k\rangle-\alpha \vert  f(0)\vert^2.
\ee
Combining this with the fact that $x'_k(0)$, $x'_2(L)$ are well defined  (since $x', x''\in  L^2$), we deduce by an integration by parts that
$$x_1'(0)+x'_2(0)-x_2'(L)=\alpha f(0).$$
Summing up, we have proved that $(x,y)\in {\mathcal D}(\AA)$ and $\AA (x,y)=(f,g)$. This gives the item (2).

Now we prove (3). 
Let $U=(u,v) \in {\mathcal D}(\mathcal{A}_\alpha)$,  and let $F=(f,g)\in {\mathcal D}(\mathcal{A}_{-\alpha}) $.   Using  the fact  that  $U$,  $F$ satisfies 
$(A_\alpha)$   and $(A_{-\alpha})$
 respectively, 
an integration by parts shows that
$$\langle {\mathcal A}_\alpha U, F\rangle=\langle (v, u^{(2)}), (f,g)\rangle=-\langle (u, v), (g,f^{(2)})\rangle=-\langle U,  {\mathcal A}_{-\alpha }F\rangle,$$
which yields
 $${\mathcal D}({\mathcal A}_{-\alpha})\subset {\mathcal D}(\mathcal{A}_\alpha^*),\,\, \text { and } {\mathcal A}_\alpha^*=-{\mathcal A}_{-\alpha} \text { on }  {\mathcal D}({\mathcal A}_{-\alpha}).$$
 Next, let $G=(\omega, \Omega) \in  {\mathcal D}(\mathcal{A}_\alpha^*)$, and put $F=(f, g):= \mathcal{A}_\alpha^*G$. By definition of the adjoint-operator, we have
 $$ \langle \AA  U, G\rangle=\langle U, \AAA G\rangle,\,\,\,\,\, \text { for all } U\in \DD(\AA). $$
 Thus
\be\label{U3}
 \sum_{k=1}^2\Big\{\langle v'_k, \omega_k'\rangle+\langle u_k^{(2)}, \Omega_k\rangle \Big\}=\sum_{k=1}^2\Big\{\langle u'_k, f_k'\rangle+\langle v_k, g_k\rangle\Big\}, \ee
for all $u_k\in H^2(R_k), $ $v_k\in H^1(R_k)$. Applying the above equality to $u_1=u_2=v_2=0$   and $v_1\in C^\infty_0(R_1)$ resp. ($u_1=u_2=v_1=0$ and $v_2\in C^\infty_0(R_2)$),  we deduce that
\be \label{U4}
-\omega''_k=g_k,\,\, \,\,\, \text { in } \DD'(R_k),\,\, k=1, 2.
\ee
Similarly, repeated  application of \eqref{U3} to  $v_1=v_2=u_k=0$ for $k=1$ and $k=2$   gives  (in the sense of distribution)
\be \label{U5}
-\Omega'_k=f'_k+C_k,\,\, \,\,\, \,\, k=1, 2,
\ee
for some constant $C_k$. Since $\Omega_1, f_1,  f'_1\in L^2(0,+\infty)$, we see that $C_1=0$, hence that
\be \label{U05}
-\Omega_1=f_1.
\ee
Making use of \eqref{U5}, and using the fact  that $f\in  H^1(\mathcal R)$ as well as  the fact that $R_2$ is a closed curve we deduce  
\be\label{U6}
\Omega_2(L)=\Omega_2(0)=\Omega_1(0).
\ee

An integration by parts using \eqref{U4} and \eqref{U5}  shows that \eqref{U3} can be written as
$$
 \big[v_1, \omega_1'\big]_0^{+\infty}+ \big[v_2, \omega_2'\big]_0^{L} +\big [u_1', \Omega_1\big]_0^{+\infty} + 
 \big [u_2', \Omega_2\big]_0^{L}=0,
 $$
which together with  \eqref{U6} and the fact that $(u,v)\in \AA$ yields
$$\Big(\omega_2'(L)-\omega_1'(0)-\omega_2'(0)\Big) v_1(0)=\Big(u_1'(0)+u_2'(0)-u_2'(L)\Big) \omega_1(0)=\alpha v_1(0)\omega_1(0).$$
Consequently,
$$\omega_1'(0)+\omega_2'(0)-\omega_2'(L)=-\alpha \omega_1(0).$$
Combining \eqref{U4}, \eqref{U5}, \eqref{U05}  and the above equality we deduce  that $\DD (\AAA)\subset \DD({\mathcal A}_{-\alpha})$.
This concludes the proof of part (1) of  Theorem \ref{th1}.

\end{proof}


{\rm  From now on, we assume that $\alpha\geq 0$. In order to  investigate the spectrum of $\AA$ in $\mathbb C^+$, we use the general standard method for dissipative operator based on the 
the Lax-Milgram theorem. First, 
 let us introduce  the linear  operator ${\mathbb G}_\alpha(z)$  
from 
$H^1(\mathcal R)$    into $ H^{-1}(\mathcal R)$ defined by
$$\langle {\mathbb G}_\alpha(z) \phi, \psi\rangle:=-\langle \phi',\psi'\rangle_{L^2({\mathcal R})}-\alpha z \overline{\psi_1(0)} \phi_1(0)-z^2\langle \phi,\psi\rangle_{L^2(\mathcal R)}, \,\,  \forall \phi \in {H}^1(\mathcal R), \psi \in H^1(\mathcal R).$$
Here $H^{-1}(\mathcal R)$ is the 
anti-dual space of $H^1(\mathcal R)$, i.e.,
 $$H^{-1}(\mathcal R):=\{\ell :H^1(\mathcal R)\rightarrow \mathbb C: \ell \text { antilinear }, \text { bounded } \}.$$
Notice that $\delta : \phi \mapsto \overline{ \phi_1(0)},$ belongs to $H^{-1}(\mathcal R)$. }

\begin{lemma}\label{P2}
For $\alpha\geq 0$ and $z\in \mathbb  C^+$, the   operator ${\mathbb G}_\alpha(z)$ is bounded and boundedly  invertible. Moreover, for all $h=(h_1,h_2)\in L^2(\mathcal R)$ and all $\nu \in \mathbb C$ the function
$w=(w_1,w_2):={\mathbb G}^{-1}_\alpha(z)\big(h+\nu \delta\big) \in H^1(\mathcal R)$ satisfies:
\be\label{UI1}
w''_k-z^2w_k=h_k,\,\, \text { in } L^2(R_i),
\ee

\be
\label{UI2}
w_1'(0)+w'_2(0)-w_2'(L)-z\alpha w_1(0)=\nu.
\ee
In particular, if $f\in H^1(\mathcal R)$ and $g\in L^2(\mathcal R)$ then $U=(u,v)\in \mathcal D(\AA)$ and
\be
\label{UI02}
(\AA-z)U=(f,g),
\ee
where
$$
u=\mathbb G_\alpha^{-1}(z)\big(zf+g+\alpha f_1(0) \delta\big),\,\,\,\,\,\,\,\, v=f+zu.$$
\end{lemma}

\begin{proof}
The first part of Lemma \ref{P2}  follows from the Lax-Milgram theorem and the following  obvious 
inequality: for all $ \phi\in H^1(\mathcal R)$
$$ \Re\Big(\langle e^{i(\pi-\arg(z))} {\mathbb G}_\alpha(z) \phi, \phi \rangle\Big)=\cos(\arg(z))
\sum_{k=1}^2 \Vert \phi_k'\Vert_{L^2(R_k)}^2$$
$$+\alpha \vert z\vert  \vert {\phi (0)}\vert^2  +\vert z\vert^2 \cos(\arg(z))  \Vert\phi\Vert_{L^2(\mathcal R)}\geq \cos(\arg(z)) \, \min(1, \vert z\vert^2) \Vert \phi\Vert_{H^1(\mathcal R)}^2.$$



 By definition of $w$, we have
\be \label{UI3}
\langle {\mathbb G}_\alpha(z) w, \phi \rangle=
-\sum_{k=1}^2 \langle w'_k,\phi_k'\rangle_{L^2(R_k)}-\alpha z {w_1(0)} \overline{\phi_1(0)} -z^2 \langle w,\phi\rangle_{L^2(\mathcal R)}
\ee
$$=\langle h+\nu \delta, \phi\rangle=\langle h, \phi\rangle_{L^2(\mathcal R)}+\nu \overline{\phi_1(0)}.$$

Applying the above equality to $\phi\in H^1(\mathcal R)$ with $\phi_1(0)=0$ we deduce \eqref{UI1}. 
Hence that $w_k\in H^2(R_k)$, since $w\in H^1(\mathcal R)$.
 Next, we write the right hand side of  \eqref{UI3} in terms of $w''_k$  by integrating by parts to obtain
 $$\langle {\mathbb G}_\alpha(z) w, \phi \rangle=
\langle w''-z^2w,\phi\rangle_{L^2(\mathcal R)}+\big(w'_1(0)+w'_2(0)-w_2'(L)-\alpha z {w_1(0)} \big)\overline{\phi_1(0)} $$
$$=\langle h, \phi\rangle_{L^2(\mathcal R)}+\nu \overline{\phi_1(0)},\,\,\,\, \forall \phi\in H^1(\mathcal R),$$
which together with \eqref{UI1}  yields
$$\big(w_1'(0)+w'_2(0)-w_2'(L)-z\alpha w(0)-\nu\big)\overline{\phi_1(0)}=0,\,\, \text { for all } \phi\in H^1(\mathcal R).$$
This gives \eqref{UI2}.  Finally, \eqref{UI02} follows  from \eqref{UI1} and \eqref{UI2} applied to $w=u$, $h=zf+g$ and  $\nu=\alpha f_1(0)$.
\end{proof}

Suppose that $z\in \mathbb C^+$, and let $U\in {\mathcal D}(\AA)$. From \eqref{U2}, we conclude  that 
$$\vert   \langle ({\mathcal A}_\alpha -z)U, U\rangle\vert 
\geq - \Re  \langle ({\mathcal A}_\alpha-z) U, U\rangle \geq
\alpha \vert v_1(0)\vert^2+ \Re z\Vert U\Vert_{\mathcal H}^2\geq  \Re z\Vert U\Vert_{\mathcal H}^2,$$
 hence that  $(\AA-z)$ is injective with closed range. Combining this with \eqref{UI02} and using the fact that  $H^1(\mathcal R)\times L^2(\mathcal R)$ is  dense  in $\mathcal H$ (see Lemma \ref{P1}), we obtain the following corollary:
 
\begin{corollary}\label{Res}  For $\alpha\geq 0$, we have $\mathbb C^+\subset \rho(\AA)$, and  
\be\label{0000}
\Vert (\AA-z)^{-1}\Vert\leq \frac{1}{\Re z},
\ee
uniformly on $z\in \mathbb C^+$.
\end{corollary}

We give now a complete description of the spectra of $\mathcal{A}_\alpha,$ for $\alpha \geq 0$. 
\begin{theorem} \label{th1}
For $\alpha \geq 0$, we have 
$$i\mathbb R\cup \sigma_p(\AA) \subset \sigma(\AA), $$
where $\sigma_p(\AA)=\Sigma_p\cup \Sigma_{d, \alpha}$ is the set of eigenvalues of the operator $\AA$, with
\begin{itemize}
\item  $\Sigma_p=\{z_n:=\frac{2\pi n}{L} i; n\in \mathbb Z^*\}$  is the set of embedded eigenvalues in $i\mathbb R$. 
\item $$\Sigma_{p,\alpha}=  \left \{
 \begin{array}{lr}
\emptyset \,\,\,\,\,\,\,\,\,\,\,\,\,\,\, \,\,\,\,\,\,\,\,\,\,\,\,\,\,\,\,\,  \,\,\,\,\,\,\,\,\,\,\,\,\,\,\,  \,\,\,\,\,\,\,\,\,\,\,\,\,\,\, \,\,\,\,\,\,\,\,\,\,\,\,\,\,\,\,\,\,\,\,\,\,\,\,\,\,\,\,\,\,\,\,\,\, \,\,\,\,\,\,\,\,\,\,\,\,\,\,\, \,\,\,\,\,\,\,\,\,\,\,\,\,\,\,\,\,\text { if } \alpha\in [0,1]\cup\{3\}\\
\{z_{n,\alpha}=\frac{1}{L}\ln\Big(\frac{3-\alpha}{1+\alpha}\Big)+\frac{2\pi n}{L}i;\,\, n\in \mathbb Z\},\,\,\,\,\,\,\,\,\,\,\,\,\,\,\,\,\,\, \,\,\,\,\,\,\,\,\,\,\,\,\,\,\,\,\, \text { if } \alpha \in  ]1,3[\\
\{z_{n,\alpha}=\frac{1}{L}\ln\Big(\frac{\alpha-3}{1+\alpha}\Big)+\frac{(2n+1)\pi }{L}i;\,\, n\in \mathbb Z\}\,\,\,\,\,\, \,\,\,\,\,\,\,\,\,\,\,\,\,\,\, \,\,\,\,\,\,\,\text { if } \alpha>3,
\end{array}
\right.$$
 is the set of  discrete eigenvalues in $\mathbb C^-$. 
\end{itemize}

\end{theorem}

\begin{proof}

 Let $\chi\in C^\infty_0(]0,+\infty[; [0,1])$ be supported in $[1,2]$  and $\Vert \chi\Vert_{L^2}=1.$ For $\lambda\in \mathbb R$ and $j\in \mathbb N^*$, we define
$$ u_{1,j}(x)=\frac{e^{i\lambda x}}{\sqrt{j}} \chi \left(\frac{x}{j}\right),\,\,\, u_{2,j}(x)=0,$$
and
$$U_j=\big( (u_{1,j},  u_{2,j}), i\lambda (u_{1,j}, u_{2, j})\big).$$
Clearly,  $U_j\in \DD(\AA)$.
An easy computation shows that
$$ \Vert u_{1,j}\Vert_{L^2(R_1)}=1, \,\,\,\,\,\,\,   \Vert u'_{1,j}\Vert_{L^2(R_1)}=\vert \lambda\vert +{\mathcal O}\left(\frac{1}{j}\right), \,\,\,\, \, \,\,\,  \Vert u''_{1,j}+\lambda^2u_{1,j}\Vert_{L^2(R_1)}={\mathcal O}\left(\frac{1}{j}\right),$$
hence that 
$$\Vert U_j\Vert_{\mathcal H}^2=\lambda^2+{\mathcal O}\left(\frac{1}{j}\right),\,\,\,\,\,\,\,\,\,\,\,  \Vert (\AA -i\lambda) U_j\Vert^2_{\mathcal H}={\mathcal O}\left(\frac{1}{j}\right), $$
and finally that  $i\mathbb R^*\subset \sigma(\AA)$, by the Weyl criterion.  This shows that $ i\mathbb R\subset \sigma(\AA)$,  since $\sigma(\AA)$ is closed.

\medskip

Next we study the point spectrum  of $\AA$.  Let $z$ be an eigenvalue of $\AA$ with eigenfunction  $U=(u,v)\in \DD(\AA)$ :
$$(\AA -z)U=0.$$
We have
\be\label{U08}
u''_k(x)-z^2u_k(x)=0,\,\, v_k=zu_k,\,\, k=1,2.
\ee


For instance,  assume that $z\in \overline{{\mathbb C^+}}:=\{z\in \mathbb C;  \Re z\geq 0\}$.  We may take $z$  to be non-zero, since the unique solution  in $\DD(\AA)$ of \eqref{U08} for $z=0$  is $(u,v)=(0,0)$.

\medskip

 Let $u_k=C_k^1 e^{-zx}+C_k^2 e^{zx}$ be a solution of \eqref{U08}. Since  $u_2(0)=u_1(0)=u_2(L)$,  $u_1'\in L^2(0, +\infty)$ and $z\not =0$, we obtain
\begin{equation}\label{U8}
 \left \{
\begin{array}{lr}
 C_1^2=0\\
C_2^1+C_2^2-C_1^1=0\\
C_2^1 e^{-zL}+C_2^2 e^{zL}-C_1^1=0.
\end{array}
\right .
\end{equation}
Next,  the transmission condition \ref{A1} applied to $(u,v)$ gives
\begin{equation}\label{U9}
C_2^1(e^{-zL}-1) +C_2^2(1-e^{zL})-(1+\alpha)C_1^1=0.
\ee
Therefore, the system \eqref{U8}, \eqref{U9}   has   solution $(C_1^1, C_1^2, C_2^1, C_2^2)\not=(0,0,0,0)$ if and only if 
\be
{\rm det} \begin{pmatrix}
1&1&-1\\
e^{-zL}&e^{zL}&-1\\
e^{-zL}-1& 1-e^{zL}& -(1+\alpha)\\
\end{pmatrix}=0, \ee
which is equivalent to 
$$ (e^{zL}-1)\left((\alpha+3)e^{zL}-(1-\alpha)\right)=0.$$
Since $\vert 1-\alpha \vert <\vert 3+\alpha\vert$ for  all $\alpha \geq 0$,  it follows that  $e^ {L\Re z}>\frac{\vert 1-\alpha\vert}{\vert 3+\alpha\vert} $ for all $\Re z\geq 0$. Hence,  the equation $(\alpha+3)e^{zl}-(1-\alpha)=0$ has no  solution in $\overline{\mathbb C^+}$. 

\medskip

Consequently, in $\overline{\mathbb C^+}$ the  eigenvalues of  $\AA$  are 
$$z=z_n:=\frac{2\pi n}{L}i, n\in \mathbb Z^*.$$
On the other hand, from \eqref{U8} and \eqref{U9} we deduce that the eigenfunction $U_n$ associated to $z_n$ is given by
\be\label{U10}
U_n =C_n\left(\left(0, \sin\left(\frac{2\pi n}{L}x\right)\right), z_n \left(0, \sin\left(\frac{2\pi n}{L}x\right)\right) \right),
\ee
where $C_n$ is a normalization constant.  
 
\medskip

Next, when  $z\in \mathbb C^-=\{z\in \mathbb C; \Re z<0\}$, we have to  choose $u_1(x)=C_1^2 e^{zx}$. We now apply the above  arguments again, with  $u_1(x)=C_1^1 e^{-zx}$replaced by $u_1(x)=C_1^2 e^{zx}$, to obtain: $z\in \mathbb C^-$ is an eigenvalue of $\AA$ if and only if
${\rm det}(\mathcal M_\alpha(z))=0$, where
\be\label{Matr}
\mathcal M_\alpha(z)=\begin{pmatrix}
1&-1&-1\\
1&-e^{zL }&-e^{-zL}\\
1-\alpha&1- e^{zL}& e^{-zL}-1\\
\end{pmatrix}.\ee
A simple calculus shows that
$$\big(e^{zL}-1\big )\big((\alpha+1) e^{zL}-(3-\alpha)\big)=0.$$
Let us  denote by $S(\alpha)$ the set  of solutions in $\mathbb C^-$ of 
 the above  equation. 
 Clearly $S(\alpha)=\emptyset$ if  $\alpha\in [0,1]\cup \{3\}$. When $\alpha>1$ and $\alpha \neq 3$,  we have
 $$S(\alpha)=\{z_n(\alpha) \in \mathbb C^-; n\in \mathbb Z\},$$
 where
\begin{itemize}
\item 
$$z_{n}(\alpha)=\frac{1}{L}\ln\Big(\frac{3-\alpha}{1+\alpha}\Big)+i\, \frac{2\pi n}{L},\,\,\, n\in \mathbb Z,\,\,\,  \text { if } \alpha \in ]1, 3[,$$
\item
$$z_{n}(\alpha)=\frac{1}{L}\ln\Big(\frac{\alpha-3}{1+\alpha}\Big)+i\,\frac{(2n+1)\pi }{L}, \,\,\,\,n\in \mathbb Z ,\,\,\, \,\,\,\,\,\,\text { if } \alpha>3.
$$
\end{itemize}

A trivial verification shows that the  eigenfunction  $U_n= C_n(u_{1,n}, u_{2,n}, z_n(\alpha)u_{1,n},z_n(\alpha)u_{2,n})$ associated to $z_n(\alpha)$ is given  by (subject to normalization constant $C_n$) 

\be
\label{E4}
u_{1, n}(x)= e^{z_n(\alpha) x},
\ee
\be\label{E6}
u_{2,n}(x)=\frac{3-\alpha}{4}\, e^{-z_n(\alpha)x}+\frac{\alpha+1}{4}\,  e^{z_n(\alpha) x}.
\ee


Conversely, we can check that the vector  defined by \eqref{U10} (resp. \eqref{E4} and \eqref{E6})   belongs to $\DD(\AA)$   and is an eigenvector corresponding to the eigenvalue $z_n$ (resp. $z_n(\alpha)$).
Summing up,  we have proved:  $\Sigma_p\cup \Sigma_{p, \alpha} =\sigma_p(\AA)$, hence that $i\mathbb R\cup \Sigma_p\cup \Sigma_{p, \alpha}\subset \sigma(\AA)$. 
   This ends the proof of the theorem.

\end{proof}

According to Corollary \ref{Res} and Theorem \ref{th1},  the proof of Theorem \ref{Intr} is completed by showing that $\mathbb C^-_{p, \alpha} :=\mathbb C^-\setminus{\Sigma_{p,\alpha}}\subset \rho(\AA)$.

Let us introduce the function
\begin{equation} 
{\mathbb  H}_\alpha(z)=\left\{
\begin{array}{l}
\vert (e^{zL}-1)\,  (3-\alpha-(1+\alpha) e^{zL})\vert, \text { if } \alpha\in ]1,3[\cup ]3, +\infty[,\\
\vert 1-e^{zL}\vert,\,\,\,\,\,\,\,\,\,\,\,\,\,\,\,\,\,\,\,\,\,\,\,\, \,\,\,\,\,\,\,\,\,\,\,\,\,\,\,\,\,\,\,\,\,\,\,\,\,\,\,\,\,\,\,\,\, \,\,\text { if } \alpha \in [0, 1]\cup \{3\}.
\end{array}
\right.
\label{ESTX}
\end{equation}
\begin{theorem}\label{resolv}
Assume that $\alpha\geq 0$.  Then for $z\in  \mathbb C^-_{p, \alpha} $  we have $z\in \rho(\AA)$,  and  for all $M>0$ there exists a constant $C=C_{\alpha, M} >0$ such that 
\be\label{Es1}
\Vert (z-\AA)^{-1}\Vert_{{\mathcal L}(\mathcal H)}\leq  \frac{C\,  {\rm max}(1, \vert z\vert^{-1})}
{{\vert \Re z\vert}\, 
{\mathbb H}_\alpha(z)}
\ee
uniformly on $z\in \mathbb C^-_{p, \alpha}$ and $\vert z\vert \leq M$. 
\medskip 

We have in particular that $\sigma (\mathcal{A}_\alpha) = i \mathbb{R} \cup \ds \Sigma_p \cup \ds \Sigma_{p,\alpha}, \, \alpha \geq 0.$
\end{theorem}

\begin{proof}
From now on, we denote $\nu:=Lz$.
For $U=(u,v) \in {\mathcal D}(\mathcal{A}_\alpha)$  and $F=(f,h)\in \mathcal H$, we have
\be\label{Resol}
(\AA-z)U=F \Longleftrightarrow
 \left \{
 \begin{array}{lr}
(\partial^2-z^2)u=h+zf,\\
\,\,\,\,\,\,\,\,\,\,\,\,\,\,\,\,\,\,\,\,\,\,\,\,v=f+zu.
\end{array}
\right.\ee

 Let $u=(u_k)_{k=1,2} \in H^1(\mathcal R)$, $u_k\in H^2(R_k) $ solve 
\be\label{U17}
(\partial^2-z^2)u=h+zf=:g.
\ee 

 Clearly, in the sense of distribution we have :
\[
u_k''-z^2 u_k= g_k \hbox{ in } R_k, \,\, k=1,2.
\]
For $0\not=z\in \mathbb C^-$ and $g_i\in L^2(R_i)$ we introduce
$$ K_z*g_1(x)=\int_0^{+\infty} K_z(x-y)g_1(y)dy,\,\, K_z*g_2(x)=\int_0^L K_z(x-y) g_2(y)dy,$$
where $K_z(y)=\frac{e^{z\vert y\vert}}{2z}$  is the Green function for the  one dimensional Helmholtz equation. For simplicity of notation,  we write
$$\beta(g_1)=K_z*g_1(0)=\int_0^{+\infty} g_1(y)\frac{e^{z\vert y\vert}}{2z}dy,\,\,\gamma_\pm(g_2)=\int_0^{L} g_2(y)\frac{e^{\pm z y}}{2z}dy.$$
Notice that
$$\gamma_+(g_2)=K_z*g_2(0),\,\, \gamma_-(g_2)= e^{-zL} K_z*g_2(L).$$
In an appendix we give some well known results on $K_{-z}*g_i$.  By Lemma \ref{Ap2}, we have
\be\label{U20}
(K_z*g_i)'(0)=-zK_z*g_j(0),\,\,\,\,\,\,\,\,\, i=1,2, 
\ee
and
\be\label{U21}
 (K_z*g_2)'(L)=z K_z*g_2(L).
\ee

We  look for  solution to \eqref{U17}  which has the form 
\be\label{U12}
u_1(x)=Ae^{zx}+De^{-zx}+ K_z*g_1(x),
\ee
$$u_2(x)=B e^{zx}+C e^{-zx}+K_z*g_2(x).$$
Since $u_1\in L^2([0,+\infty[)$ and $\Re z<0$, we may assume that $D=0$.

Using \eqref{U12}, \eqref{U20} and \eqref{U21} we see by a straightforward computation that the condition
$$u'_1(0)+u'_2(0)-u_2'(L)=\alpha f(0)+z\alpha u_1(0)$$
is equivalent to
\be\label{U22}
(1-\alpha) A+(1-e^{\nu })B+(e^{-\nu}-1)C =(1+\alpha)\beta(g_1)+\gamma_+(g_2)-\gamma_-(g_2)e^{\nu}+\frac{\alpha}{z} f(0).
\ee

\medskip

On the other hand, the  continuity condition,  $u_1(0)=u_2(0)=u_2(L)$,   yields
\be\label{U24}
A-B-C=\gamma_+(g_2)-\beta(g_1)
\ee
and
\be
A- e^{\nu }B- e^{-\nu}C= \gamma_-(g_2)e^{\nu }-\beta(g_1).
\ee

Therefore $(A, B, C)$ satisfies
$${\mathcal M}_\alpha(z)  \begin{pmatrix}
A\\
B\\
C\\
\end{pmatrix}= \begin{pmatrix}
\gamma_+(g_2)-\beta(g_1)\\
\gamma_-(g_2)e^{zL}-\beta(g_1)\\
(1+\alpha)\beta(g_1)+\gamma_+(g_2)-\gamma_-(g_2)e^{zL}+\frac{\alpha}{z} f(0)\\
\end{pmatrix},$$
where ${\mathcal M}_\alpha(z)$ is given by \eqref{Matr}.

A simple calculus gives,

\be\label{Matrbis}
 \begin{pmatrix}
A\\
B\\
C\\
\end{pmatrix}
=\frac{1}{{\rm det} ({\mathcal M}_\alpha(z))}
 \begin{pmatrix}
e^\nu+e^{-\nu}-2&e^\nu +e^{-\nu}-2&e^{-\nu}-e^{\nu}\\
1-e^{-\nu}(2-\alpha)&e^{-\nu}-\alpha&e^{-\nu}-1\\
1-\alpha e^\nu &e^{\nu}+\alpha-2& 1-e^{\nu}\\
\end{pmatrix} \ee
$$
\,\,\,\,\,\,\,\,\,\,\,\,\,\,\,\,\,\,\,\,\,\, \,\,\,\,\,\,\,\,\,\,\,\,\,\,\,\,\,\,\,\,\,\, \,\,\,\,\,\,\,\,\,\,\,\,\,\,\,\,\,\,\,\,\,\, \,\,\,\,\,\,\,\,\,\,\,\,\,\,\,\,\,\,\,\,\,\, \begin{pmatrix}
\gamma_+(g_2)-\beta(g_1)\\
\gamma_-(g_2)e^{\nu}-\beta(g_1)\\
(1+\alpha)\beta(g_1)+\gamma_+(g_2)-\gamma_-(g_2)e^{\nu}+\frac{\alpha}{z} f(0)\\
\end{pmatrix}.
$$


From  \eqref{U12}, \eqref{AAX1}, \eqref{AAX2}  and \eqref{AXXX1}, we deduce that
\be\label{XX1}
u'_1(x)=z(A-\frac{f(0)}{ 2 z}) e^{zx} -K_z*h_1(x)+\int_0^xe^{z(x-y)} h_1(y) dy+zK_z* f'_1(x),\ee
and
\be \label{XX2}
u'_2(x)=zBe^{zx}-K_z*h_2(x)+\int_0^xe^{z(x-y)} h_2(y) dy+zK_z*f'_2(x).\ee

 By \eqref{Matrbis} we have
 \be\label{XX3}
A=\frac{\alpha f(0)(1+e^\nu)}{z(3-\alpha-(1+\alpha)e^\nu)}+\frac{2\gamma_+(g_2)-2\gamma_-(g_2) e^{2\nu}+\beta(g_1) \left((3+\alpha)e^{\nu}+\alpha-1\right)}{3-\alpha-(1+\alpha) e^\nu},
\ee

\be\label{XX4}
B=\frac{\alpha f(0)}{z(3-\alpha-(1+\alpha)e^\nu)}+
\ee
$$
\frac{e^{-\nu}(\alpha-1) \gamma_+(g_2)+(2e^{-\nu}+\alpha-3) \beta(g_1)+(1-\alpha)\gamma_-(g_2) e^{\nu}}
{(e^{-\nu}-1) (3-\alpha-(1+\alpha) e^\nu)},
$$

and

\be\label{XX5}
C=\frac{\alpha f(0)e^\nu}{z(3-\alpha-(1+\alpha)e^\nu)}+
\ee
$$\frac{((2-(\alpha+1)e^\nu)\gamma_+(g_2)+2(1-e^\nu) \beta(g_1) +(2e^\nu+\alpha-3)e^\nu \gamma_-(g_2)}
{(e^{-\nu}-1)(3-\alpha-(1+\alpha) e^\nu)}.
$$
 This shows that the equation \eqref{Resol} has a unique solution in $\mathcal D(\AA)$ for all $F=(f,h)\in \mathcal H$ and all $z\in \mathbb C^-_{p, \alpha}$. Therefore $ \mathbb C^-_{p, \alpha}\subset \rho(\AA)$, since $(\AA, \mathcal D(\AA))$ is a closed operator.

In order to get the inequality \eqref{Es1}, it
will be necessary to  express $A, B, C$  with regard to $f_1', f_2', h_1$ and $h_2$.
Since $f=(f_1, f_2)\in {\widehat H}^1(\mathcal R)$,  we may assume that $f_1(0)=f_2(0)=:f(0)=0$. 
Using this fact and the fact that $g_i=h_i+zf_i$, we deduce from 
\eqref{XX3}, \eqref{XX4}, \eqref{XX5}, \eqref{A152}, \eqref{A153} and \eqref{A154} that

 
$$X=\sum_\pm G_{X,\pm}(\nu, \alpha)\gamma_\pm(h_2)+\sum_\pm K_{X, \pm}(\nu, \alpha)\gamma_\pm(f'_2)+G_{X, 1}(\nu, \alpha)\beta(h_1)+K_{X, 1}(\nu, \alpha)\beta(f'_1),$$
where $X$ indicates either $A, B$ or $C$. The constants $G_{X, i}(\nu, \alpha), K_{X, i}(\nu, \alpha), i=0,1, \pm $ depend only on $\nu=Lz$, $\alpha$ and satisfy

\be\label{XXX00}
G_{X, i}(\nu, \alpha), K_{X, i}(\nu, \alpha)
={\mathcal O}_{M, \alpha}\Big({\mathbb H}_\alpha(z)^{-1}\Big),
\ee
uniformly on $z\in  \mathbb C_{p, \alpha}^-$ and $\vert z\vert\leq M$. 
For the definition of $\mathbb H_\alpha(z)$
see \eqref{ESTX}.

Combining \eqref{U12}, \eqref{AXX2} and \eqref{AXX3}, and using the fact that $f(0)=0$,  we get
\be\label{XXX000}
v_1(x)=f_1(x)+zu_1(x)=
zAe^{zx}+zK_z*h_1(x)-K_z*f_1'(x)+z\int_0^x e^{z\vert x-y\vert} f'_1(y)dy,
\ee
\be\label{XXX0000}
v_2(x)=f_2(x)+zu_2(x)=
\ee
$$zBe^{zx}+zCe^{-z x}+zK_z*h_2(x)-K_z*f_2'(x)+z\int_0^x e^{z\vert x-y\vert} f'_2(y)dy.$$

Now using   \eqref{XX1}, \eqref{XX2}, \eqref{XXX00}, \eqref{XXX000}, \eqref{XXX0000},  Lemma \ref{Ap1} and the following inequality
$$\left|\int_0^xe^{z(x-y)} h_1(y) dy\right|\leq \int_0^{+\infty}e^{\Re(z) \vert x-y\vert} \vert h_1(y) \vert \, dy,$$
we obtain
$$\sum_{k=1,2}\Vert v_k\Vert_2+\Vert u'_k\Vert_2\leq  \frac{C\,  {\rm max}(1, \vert z\vert^{-1})}
{{\vert \Re z\vert}\, 
{\mathbb H}_\alpha(z)}\sum_{k=1,2}(\Vert h_k\Vert_2+\Vert f'_k\Vert_2),$$
which yields \eqref{Es1}.


\end{proof}

\begin{remark}
Using  \eqref{Resol}, \eqref{U12} and \eqref{Matrbis},  one  obtain explicitly the kernel  $\mathbb K(x,y;z) $ of $(\AA-z)^{-1}$.

\end{remark}

\section{Riesz basis of the stabilization subspace} \label{asybasis}



In this section, it is proved that the generalized eigenfunctions of the dissipative operator $\mathcal{A}_\alpha$ associated to the eigenvalues in $\Sigma_{p,\alpha}, \alpha \in ]0,1[ \cup ]3,+\infty[,$ form a Riesz basis of the subspace of $\mathcal{H}$ which they span (denoted by $\mathcal{H}_{p,\alpha}$). 
To this end, we recall that a sequence $(\Psi_n)_{n\in\ZZ}$ is a Riesz basis in a Hilbert space $V$ if there exist a Hilbert space $\tilde V,$ an orthonormal basis $(e_n)_{n\in\ZZ}$ of $\tilde{V}$ and an isomorphism $\Theta : \tilde{V}\longrightarrow V$ such that $\Theta e_n=\Psi_n, \forall  n\in \ZZ.$ Our strategy is to prove that the Gram matrix $(\langle \Psi_n, \Psi_m \rangle)_{n, m\in\ZZ}$ is a bounded bijective operator on $\ell^2(\ZZ).$ This is equivalent to prove the existence of two positive constants $A$ and $B$ such that
$$
A\sum_n|c_n|^2\le\left\|\sum_n c_n\Psi_n\right\|_{\mathcal{H}}^2\le B\sum_n|c_n|^2
$$
for every sequence $(c_n)_{n\in \ZZ}\in \ell^2(\ZZ).$
\medskip 
 
We start by computing the normalized  eigenfunctions of $\mathcal{A}_\alpha$ corresponding to the eigenvalues $z_n(\alpha).$  
We assume throughout this section that $\alpha \in ]1,3[.$ The case when $\alpha\in ]3,+\infty[$ can be studied by the same approach.

So let $n$ and $\alpha \in ]1,3[ $ be fixed. In order to simplify the writings, we put $\omega_\alpha=\frac{1}{L}\ln\left( \frac{1+\alpha}{3-\alpha}\right)$ and $\nu=\frac{2\pi}{L}.$ By this doing we obtain $z_n(\alpha)=-\omega_\alpha+in\nu.$

In vertue of (\ref{E4}) and (\ref{E6}), the eigenfunction associated to $z_n(\alpha)$ is of the form given by
\be\label{eigfunc}
\Psi_n=C_n(u_{1,n},u_{2,n}, z_n(\alpha)u_{1,n},z_n(\alpha)u_{2,n})
\ee
with 
$$
u_{1,n}(x) =  e^{-\omega_\alpha x}e^{in\nu}, \quad \text{for}\; x\in[0,+\infty[$$
and
$$ u_{2,n}(x) = \frac{3-\alpha}{4}e^{-\omega_\alpha x}e^{in\nu} +\frac{1+\alpha}{4}e^{\omega_\alpha x}e^{-in\nu} , \quad \text{for}\; x\in[0,L].$$
where $C_n$ is a constant such that $\|\Psi_n\|_{\mathcal{H}}=1.$
\begin{remark}\label{symetry}
    Some useful symmetry properties of the eigenvalues and the eigenfunctions are below.
    \begin{itemize}
        \item $\forall n\in\ZZ, \; z_{-n}(\alpha)=\overline{z_n(\alpha)}.$
        \item $\forall n\in\ZZ, \; u_{1,-n}(x)=\overline{u_{1,n}(x)}.$
         \item $\forall n\in\ZZ, \; u_{2,-n}(x)=\overline{u_{2,n}(x)}.$ 
    \end{itemize}
\end{remark}
\begin{proposition}
    Let $n\in\ZZ$ and $C_n>0$  defined by 
    $$ C_n^{-2}=\left( \omega_\alpha+\left(\frac{\nu}{\omega_\alpha}\right)^2n^2\right)\left(1+\left( \frac{3-\alpha}{4}\right)^2(e^{2\omega_\alpha L}-1)+ \left( \frac{1+\alpha}{4}\right)^2(1-e^{-2\omega_\alpha L})\right).$$
    Then, the eigenfunction $\Psi_n$ is normalized in $\mathcal{H},$ i.e. $\|\Psi_n\|_{\mathcal{H}}=1.$ \\
    Moreover, there exists $\tilde{C}>0$ such that, for all large $n\neq m\in\ZZ, $ we have
    \be
    \left|\langle\Psi_n,\Psi_m\rangle_\mathcal{H}\right|={\mathcal{O}}\left(\frac{\tilde{C}}{|n-m|}\right). 
    \ee
\end{proposition}
\begin{proof}
    First, a straightforward calculus gives
    $$
    \|u_{1,n}\|_{L^2(0,\infty)}^2=\frac{1}{2\omega_\alpha},\; \|u_{1,n}'\|_{L^2(0,\infty)}^2=\frac{\omega_\alpha^2+\nu^2 n^2}{2\omega_\alpha},
    $$
    and 
    $$
    \|u_{2,n}\|_{L^2(0,L)}^2=\frac{1}{2\omega_\alpha}\left\{\left(\frac{3-\alpha}{4}\right)^2(e^{2\omega_\alpha L}-1)+ \left( \frac{1+\alpha}{4}\right)^2(1-e^{-2\omega_\alpha L}) \right\}, $$
    and 
    $$
    \|u_{2,n}'\|_{L^2(0,L)}^2=(\omega_\alpha^2+\nu^2n^2)\|u_{2,n}\|_{L^2(0,L)}^2.
    $$
    Thus the normalizing constant $C_n$ can be obtained.
    Now, let $n, m\in \ZZ$  and assume $n\neq m.$ One can compute the inner products $\langle u_{1,n},u_{1,m}\rangle, \langle u'_{1,n},u'_{1,m}\rangle, \langle u_{2,n},u_{2,m}\rangle$ and $\langle u'_{1,n},u'_{2,m}\rangle$ explicitly since all of the functions are given. We obtain
    \be\label{gram}
\langle \Psi_n, \Psi_m \rangle_\mathcal{H}=\left\{ \begin{array}{lcl}
 \left(1+K \right)  \frac{C_nC_mz_n\bar{z}_m}{\omega_\alpha+i(n-m)\nu}   & if & n\neq -m  \\
  \left(1+K -\frac{(1+\alpha)(3-\alpha)L}{8}\right)  \frac{C_n^2|z_n|^2}{\omega_\alpha+i\nu}   & if & n=-m
\end{array} \right.
    \ee
    where we have set $$K=\left(\frac{3-\alpha}{4}\right)^2(e^{2\omega_\alpha L}-1)+ \left( \frac{1+\alpha}{4}\right)^2(1-e^{-2\omega_\alpha L}).$$
    Recalling that $C_n$ behaves like ${\mathcal{O}}(|n|^{-1})$  and that $|z_n|$ is of order $|n|$ for large $n$, we obtain (\ref{gram}).
\end{proof}

\begin{theorem} \label{rieszbasis} {\rm [}Riesz basis for the operator $\mathcal{A}_\alpha${\rm ]} 
The sequence of generalized eigenfunctions of $\mathcal{A}_\alpha, \alpha \in ]1,3[ \cup ]3,+\infty[,$ forms a Riesz basis of $\mathcal{H}_{p,\alpha}.$
\end{theorem}
\begin{proof}
We assume that $1<\alpha<3$ and we recall that $\mathcal{H}_{p,\alpha}=\overline{{\rm{Span}}\left\{\Psi_n; n\in\ZZ\right\}}.$
\begin{itemize}
    \item The sequence $\left\{\Psi_n; n\in\ZZ\right\} $ is complete in  $\mathcal{H}_{p,\alpha}.$
    \item We consider the map $ \Theta : \ell^2(\ZZ)\rightarrow  \mathcal{H}_{p,\alpha}$ defined by 
    $$ \Theta((c_n)_{n\in \ZZ})=\sum_{n\in \ZZ}c_n \Psi_n.$$
    We claim that $\Theta$ is an isomorphism. Let $N\in\NN, p\in\NN.$ We have 
    \begin{eqnarray}
         \left\| \sum_{n=N}^{N+p}c_n\Psi_n\right\|^2-\sum_{n=N}^{N+p} |c_n|^2&=& \sum_{j=N}^{N+p}\sum_{k=j+1}^{N+p}2\Re\left( c_j\bar{c}_k \langle \Psi_j,\Psi_k \rangle\right) \label{conti}\\
         &\le& 2\tilde{C}\sum_{m=0}^{+\infty}|c_m|\sum_{n=1}^{+\infty}\frac{|c_{n+m}|}{n}\nonumber\\
         &\le& A\sum_{n=0}^{+\infty}|c_n|^2, \; \text{for some constant}\; A>0.\nonumber
    \end{eqnarray}
   As a consequence of the properties given in Remark \ref{symetry}, the same can be  done for negative values of $N\in\ZZ.$ Thus, the series $\ds \sum_{n\in \ZZ}c_n \Psi_n $ converges in $\mathcal{H}$ and there exists $\tilde{K}>0$ such that $\|\Theta((c_n)_{n\in \ZZ})\|_{\mathcal{H}}\le \tilde{K} \ds \sum_{n\in\ZZ}|c_n|^2.$ Since $\Theta$ is linear, this proves that $\Theta$ is continuous.
   \item The map $\Theta $ is one to one.
   \item Moreover, using again (\ref{conti}), we see for $N$ large enough 
   $$
   \left\| \sum_{n=N}^{+\infty}c_n\Psi_n\right\|^2-\sum_{n=N}^{+\infty} |c_n|^2 \le \frac{1}{2}\sum_{n=N}^{+\infty} |c_n|^2.
   $$
   This is also because the series $\sum_n\frac{1}{n^2}$ converges. Now, we fix such integer $N.$ We will prove the continuity of $\Theta^{-1}$ by contradiction. We assume that there exists a normalized sequence $(c_n)_{n\in \ZZ}\in\ell^2(\ZZ), $ and a subsequence $(c_n^k)_{k\in \ZZ}\in\ell^2(\ZZ)$ of  $(c_n)_{n\in \ZZ}$ such that 
   $$
   \lim_{k\to+\infty}\left\| \sum_{n\in\ZZ}c_n^k\Psi_n\right\|=0.
   $$
   We set $\sum_{n\in\ZZ}c_n^k\Psi_n=\tilde{\varphi}_k+\psi_k$ with
   $$
   \varphi_k=\sum_{|n|\le N-1}c_n^k\Psi_n \quad and \quad \psi_k=\sum_{|n|\ge N}c_n^k\Psi_n.
   $$
   Up to an extraction of a subsequence we may assume that $(\varphi_k)$ converges in $\mathcal{H} $ to some $\varphi$ when $k$ tends to infinity. Therefore we have
   $$
   \|\varphi_k+\psi_k\|^2=\|\varphi_k\|^2+2\Re(\langle P\varphi_k, \psi_k \rangle)+\|\psi_k\|^2\ge \|\varphi_k\|^2-\|P\varphi_k\|^2
   $$
   where $P$ is the orthogonal projection on ${\rm{Span}}\{\varphi_m; |m|\ge N\}.$ For $k\to+\infty$ we have $\|\varphi_k\|\to \|\varphi\|$ and $\|P\varphi_k\|\to \|P\varphi\|.$ This yields  $\varphi=P\varphi$ and consequently $\varphi=0.$ 
   This is impossible since 
   $$
    0<\frac{1}{2}\sum_{n=N}^{+\infty} |c_n|^2\le \lim_{k\to+\infty}\left\| \sum_{n\in\ZZ}c_n^k\Psi_n\right\|=0.
   $$
   This ends the proof.
\end{itemize}
\end{proof}

\section{Energy estimate} \label{estenerg}

The Lumer-Phillips theorem combined with Theorem \ref{th1} implies that the operator $\mathcal{A}_\alpha, \alpha \geq 0,$ generates a $C_0$ semigroup of contractions on $\mathcal{H}$, denoted by $(e^{t\mathcal{A}_\alpha})_{t \geq 0}$. So, system (\ref{wsystem}) is well-posed in $\mathcal{H}$ (and in $\mathcal{D}(\mathcal{A}_\alpha))$. More precisely, we have the following corollary.

\begin{corollary} \label{wp} 
For all $(u_0,u_1) \in \mathcal{H}$, the wave system (\ref{pbfirstorderch})
admits a unique mild solution $(u,u^\prime) \in C([0,+\infty);\mathcal{H}).$ Moreover, if $(u_0,u_1) \in \mathcal{D}(\mathcal{A}_\alpha)$ then (\ref{pbfirstorderch}) admits a unique strong solution  $(u,u^\prime) \in C([0,+\infty);\mathcal{D}(\mathcal{A}_\alpha))$ and satisfies the energy identity (\ref{energid}).
\end{corollary}

Now, using the Riesz basis constructed in the latest section, the energy is proved to decrease exponentially to a non-vanishing value depending on the initial datum. The decay rate is explicitly given at the end of Theorem \ref{energy} below since $\alpha \in ]1,3[ \cup ]3,+\infty[.$  

\begin{theorem} \label{energy} {\rm (}Energy estimate{\rm )} 
We assume that $\alpha \in ]1,3[ \cup ]3,+\infty[$. Let $$E(t) := \frac{1}{2} \, \left\|\begin{pmatrix}
u(t) \\ u^\prime(t)
\end{pmatrix}\right\|_{\mathcal{H}}^2, \, t \geq 0,$$ be the energy, $\mathcal{H}_{p,\alpha}$ (respectively $\mathcal{H}_p$) be the subspace of $\mathcal{H}$ spanned by the $\psi^{p,\alpha}(\lambda, \cdot)$'s (resp. $\psi^p(\omega, \cdot)$'s), which are the normalized (in $\mathcal{H}$) eigenfunctions of $\mathcal{A}_\alpha$ associated to the eigenvalues $\lambda$ in $\Sigma_{p,\alpha}$ (resp. $\Sigma_p$). We have :

\begin{enumerate}
\item $\mathcal{H}_{p,\alpha}$ is orthogonal to $\mathcal{H}_p$.  
\item Let $(u_0,u_1)$ in $\mathcal{H}_{p,\alpha} \oplus \mathcal{H}_p$ be the initial condition of the boundary value problem (\ref{pbfirstorderch}) and $(u_0^p,u_1^p)$ its orthogonal projection onto $\mathcal{H}_p$. \\
Then $(u(t),u^\prime(t))$ decreases exponentially, in $\mathcal{H}$, to $(u_0^p,u_1^p)$ 
when $t$ tends to $+ \infty$. More precisely
\be 
\label{optimaldecay}
E(t) = E^p (t) + E^{p,\alpha}(t) = E^p(0) + E^{p,\alpha}(t) \leq E^p(0) + e^{-2\omega t} E^{p,\alpha}(0)
\ee 
where 
$$
E^{p}(t) := \frac{1}{2} \, \left\| 
e^{t\mathcal{A}_\alpha}
\begin{pmatrix}
u_0^{p} \\
u_1^{p}
\end{pmatrix}
\right\|_{\mathcal{H}}^2, \,
E^{p,\alpha}(t) := \frac{1}{2} \, \left\| 
e^{t\mathcal{A}_\alpha}
\begin{pmatrix}
u_0 - u_0^{p} \\
u_1 - u_1^{p}
\end{pmatrix}
\right\|_{\mathcal{H}}^2, \, t \geq 0, 
$$
and 
$$
-\omega:= \sup_{\lambda \in \Sigma_{p,\alpha}} \Re(\lambda) < 0.
$$ 
\end{enumerate}
\end{theorem}

\begin{proof} 
\begin{enumerate}
\item
 Above all, it is easy to see that if $\mathcal{A}_\alpha \psi^p(\lambda,.)=\lambda \psi^p(\lambda,.),$ then $\mathcal{A}^*_\alpha \psi^p(\lambda,.) = - \lambda \psi^p(\lambda,.).$ 

\medskip

Now, to prove that $\mathcal{H}_p$ is orthogonal to $\mathcal{H}_{p,\alpha}$, it suffices to check that any generalized 
eigenfunction $\psi^{p,\alpha}(\lambda',.)$ of $\mathcal{H}_{p,\alpha}$ is orthogonal to any eigenfunction $\psi^p(\lambda,.)$ 
of $\mathcal{H}_p.$ 

\medskip

First we assume that $\psi^{p,\alpha}(\lambda',.)$ is an eigenfunction, i.e $$ \mathcal{A}_\alpha \psi^{p,\alpha}(\lambda',.)= 
\lambda' \psi^{p,\alpha} (\lambda',.).$$ Therefore, since $\lambda$ is purely imaginary,  

$$\begin{array}{lll}
\lambda' \left\langle \psi^{p,\alpha}(\lambda',.),\psi^p(\lambda,.)\right\rangle_{\mathcal{H}} &=& \left\langle \mathcal{A}_\alpha \psi^{p,\alpha}(\lambda',.),\psi^p(\lambda,.)\right\rangle_{\mathcal{H}}\\
&=&\left\langle  \psi^{p,\alpha}(\lambda',.), \mathcal{A}^*_\alpha \psi^p(\lambda,.)\right\rangle_{\mathcal{H}}\\
&=&-\left\langle  \psi^{p,\alpha}(\lambda',.), \mathcal{A}_\alpha \psi^p(\lambda,.)\right\rangle_{\mathcal{H}}\\
&=&\lambda \left\langle \psi^{p,\alpha}(\lambda',.),\psi^p(\lambda,.)\right\rangle_{\mathcal{H}}.
\end{array}
$$

Consequently $\left\langle \psi^{p,\alpha}(\lambda',.),\psi^p(\lambda,.)\right\rangle_{\mathcal{H}}=0.$

\medskip

Secondly, we assume that $\lambda'$ is not simple. Let $\psi^{p,\alpha}(\lambda',.)$ be an associated generalized eigenfunction of order $p\geq 2,$ in the sense that

$$(\mathcal{A}_\alpha -\lambda')^p \psi^{p,\alpha}(\lambda',.)=0,\;(\mathcal{A}_\alpha-\lambda')^{p-1} \psi^{p,\alpha}(\lambda',.)\neq 0.$$

Setting $\psi=(\mathcal{A}_\alpha-\lambda') \psi^{p,\alpha}(\lambda',.),$ then $\psi$ is a generalized eigenfunction associated to $\lambda'$ of order $p-1,$ so arguing by iteration with respect to the order $p$ we can assume that $\left\langle \psi,\psi^p(\lambda,.)\right\rangle_{\mathcal{H}}=0.$ 

\medskip

Therefore
$$
\begin{array}{lll}
\lambda' \left\langle \psi^{p,\alpha}(\lambda',.),\psi^p(\lambda,.)\right\rangle_{\mathcal{H}}&=& \left\langle \mathcal{A}_\alpha\psi^-(\lambda',.)+\psi,\psi^p(\lambda,.)\right\rangle_{\mathcal{H}}  \\
&=& \left\langle \mathcal{A}_\alpha \psi^{p,\alpha}(\lambda',.),\psi^p(\lambda,.)\right\rangle_{\mathcal{H}}\\
&=&\lambda \left\langle \psi^{p,\alpha}(\lambda',.),\psi^p(\lambda,.)\right\rangle_{\mathcal{H}},
\end{array}
$$
as previously. Consequently $\left\langle \psi^{p,\alpha}(\lambda',.),\psi^p(\lambda,.)\right\rangle_{\mathcal{H}}=0.$ 

\medskip

\item
The purely point spectrum of the dissipative operator $\mathcal{A}_\alpha$ is the union of $\Sigma_p$ set of the purely imaginary eigenvalues and $\Sigma_{p,\alpha}$ set of the other eigenvalues. 

\medskip

The initial conditions $(u_0,u_1)$ are written as a sum of two terms: 

$$
U_0 := 
\begin{pmatrix}
u_0 \\ u_1 
\end{pmatrix} = 
\sum_{\lambda \in \Sigma_{p}} U^p_0(\lambda) \psi^p(\lambda, \cdot) + \sum_{\lambda \in \Sigma_{p,\alpha}} U^{p,\alpha}_0(\lambda) \psi^{p,\alpha} (\lambda, \cdot)
$$

where $\psi^p(\lambda, \cdot)$ (respectively $\psi^{p,\alpha}(\lambda, \cdot)$) is a normalized (in $\mathcal{H}$) eigenfunction of $\mathcal{A}_\alpha$ associated to the eigenvalue $\lambda$ in $\Sigma_p$ (resp. $\Sigma_{p,\alpha}$). Note that the sum takes into account the multiplicities of the eigenvalues here. 

\medskip

Thus the solution of the boundary value problem (\ref{pbfirstorderch}) is:

$$
U(t) := 
\begin{pmatrix}
u(t) \\ u^\prime(t) 
\end{pmatrix}
= \sum_{\lambda \in \Sigma_p} U^p_0(\lambda)  e^{\lambda t} \psi^p(\lambda, \cdot) + \sum_{\lambda \in 
\Sigma_{p,\alpha}} U^{p,\alpha}_0(\lambda) e^{\lambda t} \psi^{p,\alpha}(\lambda, \cdot).
$$    

The energy $E(t)= E^p(t) + E^{p,\alpha}(t)$ with 

$$E^p(t):= 1/2 \, \sum_{ \lambda \in \Sigma_p} \left| U^p_0(\lambda) \right|^2  |e^{\lambda t}|^2, 
$$
$$
E^{p,\alpha}(t):= 1/2 \, \sum_{\lambda \in \Sigma_p} \left| U^{p,\alpha}_0(\lambda) \right|^2 |e^{\lambda t}|^2.
$$
Now, since $\Sigma_p$ contains only purely imaginary eigenvalues, $|e^{2\lambda t}| = 1$, for any $\lambda \in \Sigma_p$ and any $t > 0$. Thus $E^p(t)=E^p(0)$ for any $t > 0$. 

\medskip

The real part of $\lambda$ is a non-positive real number if $\lambda$ is such that $\lambda \in \Sigma_{p,\alpha}$. This real part is proved to be equal to 
$$
\left \{
 \begin{array}{lr}
 \frac{1}{L}\ln\Big(\frac{3-\alpha}{1+\alpha}\Big), \hbox{ if } \, \alpha \in  ]1,3[, \\
\frac{1}{L}\ln\Big(\frac{\alpha-3}{1+\alpha}\Big) \hbox{ if } \alpha>3,
\end{array}
\right.
$$ 
when $\lambda \in \Sigma_{p,\alpha}$.   

\medskip

It holds $E^{p,\alpha}(t) \leq e^{-2\omega t} E^{p,\alpha}(0)$. Thus $E^{p,\alpha}(t)$ decreases exponentially to $0$ when $t$ tends to $+ \infty$ and the total energy $E(t)$ decreases exponentially to $E^p(0)$ when $t$ tends to $+ \infty$. 
\end{enumerate}

\end{proof}
 

  \section{Appendix} \label{app}
  
  In this appendix we prove Lemma \ref{P1}, and we recall some standard  results on the convolution of  functions in $H^s(R_k)$.
  
\begin{proof}[Proof of Lemma \ref{P1}]

Let $U=(u_1, u_2) \in \widehat{H}^1(\mathcal R)$, and let $\chi\in C^\infty([0, +\infty[; \mathbb R)$ with $\chi(x)=1$ for $x<1$  and $\chi (x)=0$ for $x>2$. For $R>0$, put $\chi_R(x)=\chi(\frac{x}{R})$, and $U_R=(\chi_R u_1, u_2)$.
We claim that:
$$\forall \epsilon>0 , \,\, \exists R=R_{\epsilon,u_1},\,\, \text { s.t }  \Vert (u_1-u_1X_R)'\Vert\leq \epsilon,$$
which yields the lemma since $U_R\in H^1(\mathcal R)$.
We may assume that  $\Vert u'_1\Vert\not=0$, otherwise the  claim is trivial. Applying Cauchy Schwarz inequality  to $u_1(x)=u_1(0)+\int_0^xu'_1(t)dt$, we get
\begin{equation}\label{ESS}
\vert u_1(x)\vert\leq \vert u_1(0)\vert+\sqrt{x} \Vert u'_1\Vert\leq C_{u_1} \sqrt{x},
\end{equation}
for $x\geq 1$. Writing 
$$\vert u_1(x)\vert^2-\vert u_1(t)\vert^2=\int_t^x u'_1(y)\overline{u_1(y)}+u_1(y)\overline{u'_1(y)}dy,$$
applying Cauchy Schwarz inequality to the right hand side of the above equality,  and  using \eqref{ESS} we get :
$$\vert u_1(x)\vert^2\leq \vert u_1(t)\vert^2+2C_{u_1}{x} \int_t^x\vert u'_1(y)\vert^2 dy, $$
for all $x>t\geq 1$. Next, fix  $t$ large enough such that
$2C_{u_1} \int_t^x\vert u'_1(y)\vert dy\leq \frac{\epsilon}{2}$, and let
$R>>1$ so that  $\vert u_1(t)\vert^2\leq \frac{\epsilon R}{2}$, we obtain : 
$$\vert u_1(x)\vert^2\leq \epsilon x,\,\,\,\, \forall x\in [R, +\infty[.$$
Multiplying  both sides of the above inequality by $\frac{1}{R^2}
\vert \chi'(\frac{x}{R})\vert^2$ and integrating, we get
$$\Vert\chi_R' u_1\Vert^2\leq \frac{\epsilon}{R^2} \Vert \chi'\Vert_\infty \int_R^{2R} x{dx}=\frac{3}{2} \epsilon \Vert \chi'\Vert_\infty.$$
This ends the proof of the claim since $(u_1-\chi_R u_1)'=(1-\chi_R)u_1'-\chi_R' u_1$  and
$$\lim_{R\rightarrow +\infty}\Vert(1-\chi_R)u_1'\Vert=0.$$

\end{proof}

  For $f_k\in L^2(R_k)$ and $z\in \mathbb C^-$,  we introduce
  $$K_z*f_1(x)=\int_0^{+\infty} K_-z(x-y) f_1(y)dy,\,\,\, K_z*f_2(x)=\int_0^L K_z(x-y) f_2(y).$$
Recall that, for $g \in L^2(\mathbb R)$ and $h\in L^1(\mathbb R)$, we know that $g*h(x)\in L^2(\mathbb R)$ and 
$$\Vert g*h\Vert_{L^2(\mathbb R)} \leq \Vert h\Vert_{L^1(\mathbb R)} \Vert g\Vert_{L^2(\mathbb R)}. $$ 
Applying  this result  to  $g(y)=1_{[0,+\infty [}(y) f_1(y), g(y)= 1_{[0, L]}(y) f_2(y) \in L^2(\mathbb R)$  and $h(y)=T_z(y)={e^{z \vert y\vert}}=2zK_z(y)$, we get
\begin{lemma}\label{Ap1}
$$\Vert  T_z*f_k\Vert_{L^2(R_k)}\leq  \Vert  T_z*f_k\Vert_{L^2(\mathbb R)}\leq  \Vert T_z\Vert_{L^1(\mathbb R)}  \Vert f_k\Vert_{L^2(R_k)}\leq \frac{1}{\Re z} \Vert f_k\Vert_{L^2(R_k)}.$$
\end{lemma}

\begin{lemma}\label{Ap2}
For $f_k\in L^2(R_k)$, we have
\be\label{AAX1}
\Big(K_z* f_1\Big)'(x)=-zK_z*f_1(x)+\int_0^x e^{z(x-y)} f_1(y) dy,\ee
and 
\be\label{AAX2}
\Big(K_z* f_2\Big)'(x)=zK_z*f_2(x)-\int_x ^Le^{z(x-y)} f_1(y) dy.\ee
In particular,
$$(K_z*f_1)'(0)=-zK_z*f_1(0),\,\,\,\,\,\,(K_z*f_2)'(L)=z K_z*f_2(L).$$
\end{lemma}
\begin{proof} The lemma follows from the following equality
$$\Big(K_z*f_k\Big)'(x)=z\int_0^x e^{z(x-y)} f_k(y) \frac {dy}{2z}-z\int_x^{C_k} e^{z(y-x)} f_k(y) \frac {dy}{2z}, \, k=1,2,$$
where $C_1=+\infty$ and $C_2=L$.
\end{proof}
The following lemma follows by an integration by parts.
\begin{lemma}\label{Ap3}
For $f=(f_1,f_2)\in \widehat{H}^1(\mathcal R)$, we have
\be\label{AXXX1}
\Big(K_z* f_1\Big)'(x)=K_z*f_1'(x)-\frac{f(0)}{2z} e^{zx},\ee
\be\label{AXX2}
f_1(x)+z^2 K_z* f_1(x)=-K_z* f_1'(x)+\frac{f(0)}{2} e^{zx}+z\int_0^x e^{z\vert x-y\vert} f'_1(y)dy,\ee
\be\label{AXX3}
f_2(x)+z^2 K_z* f_2(x)=\ee
$$-K_z* f_2'(x)+\frac{f(0)}{2} e^{zx}+z\int_0^x e^{z\vert x-y\vert} f'_2(y)dy+\frac{f(0)}{2}(e^{zx}+e^{\nu}e^{-zx}),$$
\begin{equation}\label{A151}
K_z* f_2'(x)=\Big({e^{\nu}e^{-zx}-e^{zx}}\Big) \frac{ f(0)}{2z}-zK_z*f_2(x)+\int_0 ^xe^{z(x-y)} f_2(y) dy, \end{equation}
\begin{equation}\label{A152}
\gamma_+(zf_2)=\frac{e^{\nu}-1}{2z}f(0)-\gamma_+(f'_2),
\end{equation}
\begin{equation}\label{A153}
\gamma_-(zf_2)=\frac{1-e^{-\nu}}{2z}f(0)+\gamma_-(f'_2),
\end{equation}
\begin{equation}\label{A154}
\beta(z f_1)=-\frac{f(0)}{2z}-\beta(f'_1).
\end{equation}
We recall that $f_1(0)=f_2(0)=f_2(L)=:f(0).$
\end{lemma}

\end{document}